\documentclass[reqno]{amsart}
\usepackage{geometry}                		
\usepackage{graphicx}				
\usepackage{amssymb, color}
\usepackage{amsmath}
\usepackage[all]{xy}
\usepackage{array} 

\theoremstyle{plain}
\usepackage{latexsym}
\usepackage{amsthm}
\usepackage{amsfonts}
\usepackage{mathrsfs}
\usepackage{enumerate}
\usepackage{pst-node}
\usepackage{tikz-cd}
\usepackage{stmaryrd}
\usepackage{epsfig}
\usepackage[sc]{mathpazo}
\usepackage{verbatim}
\usepackage{breqn}
\newtheorem{definition}{Definition}
\newtheorem{theorem}{Theorem}

\newtheorem{lemma}[theorem]{Lemma}

\newcommand{\cblue}[1]{{\textcolor{black}{#1}}} 

\newcommand{\cccblue}[1]{{\textcolor{black}{#1}}}
\newcommand{\cred}[1]{{\textcolor{black}{#1}}} 

\newcommand{\cgreen}[1]{{\textcolor{black}{#1}}}
\newcommand{\ccgreen}[1]{{\textcolor{black}{#1}}}
\title{Orientation data for coherent sheaves on the local projective plane}
\author{Yun Shi}{\thanks{The author was partially supported by the NSF Grant DMS-1440140 while the author was in residence at the MSRI in Berkeley during the Spring 2018 semester. The author was also partially supported as a research assistant through her thesis advisor's NSF Grant DMS-1802242}\thanks{Department of Mathematics, University of Illinois, Urbana, IL, 61801. yunshi2@illinois.edu}}
\date{}							

\begin{document}
\maketitle
\begin{abstract}
In this note, we show that the canonical orientation data defined in \cite{Dav} on the quiver hearts are compatible under the autoequivalence $\_\otimes\pi^*O(1)$, and hence glue to give an orientation data for the stack of coherent sheaves on the local projective plane .  
\end{abstract} 
\section{Introduction}

Orientation data is an important ingredient in the definition of Motivic Donaldson-Thomas (DT) theory \cite{BJM19} \cite{KS}. Roughly speaking, an orientation data is a square root of the virtual canonical bundle on a moduli stack satisfying some compatibility conditions. \cgreen{A choice of orientation data is necessary to define a global Motivic DT invariants \cite{BJM19}}. As illustrated in \cite{Dav}, a choice of orientation data \cgreen{satisfying the cocycle condition (see Definition \ref{def1})} is necessary for the map from the Hall algebra of stack functions to the Grothendieck ring of motivic weights to be a ring homomorphism. The existence of orientation data had been shown in many studies. For example it is shown in \cite{Dav} that there exists a canonical orientation data for the moduli stack of quiver representations, and in \cite{Tod} that the trivial bundle on the moduli stack of finite length sheaves on a CY 3-fold defines an orientation data for that moduli stack.

In this note, we consider the orientation data on local $\mathbb{P}^2$: $Tot(O_{\mathbb{P}^2}(-3))$, we denote this space by $Y$, and $\pi:Y\rightarrow \mathbb{P}^2$ be the projection. 
There is a derived equivalence \cite{Bri05}:
\begin{equation*}
\label{equ1}
RHom(\pi^*(O_{\mathbb{P}^2}\oplus O_{\mathbb{P}^2}(1)\oplus O_{\mathbb{P}^2}(2)), \_):D^b(Y)\rightarrow D^b(Mod-A)
\end{equation*}
where $A\simeq End(\pi^*(O_{\mathbb{P}^2}\oplus O_{\mathbb{P}^2}(1)\oplus O_{\mathbb{P}^2}(2)))$, and $Mod-A$ is the category of right $A$ modules. The noncommutative algebra $A$ can be described as the Jacobi algebra of a quiver with potential $(Q, W)$, i.e. $A\simeq Q_1/\partial W$. Here $Q$ is the following quiver:
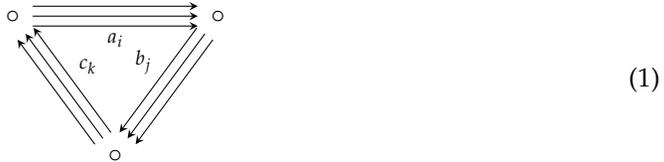
\begin{equation}
\label{equ2}
\begin{tikzcd}[arrow style=tikz,>=stealth,row sep=4em]
\circ \arrow[rr, shift left=.8ex]
\arrow[rr]
\arrow[rr, shift right=.8ex, swap, "a_i"]
&& \circ \arrow[dl,shift left=.8ex]
\arrow[dl]
\arrow[dl,shift right=.8ex, swap, "b_j"]\\
& \circ \arrow[ul,shift left=.8ex]
\arrow[ul]
\arrow[ul,shift right=.8ex, swap, "c_k"]
\end{tikzcd}
\end{equation}
and $Q_1$ is the path algebra of $Q$. The potential is given by $W = c_3b_2a_1-c_2b_3a_1+c_1b_3a_2-c_3b_1a_2+c_2b_1a_3-c_1b_2a_3$. 

Let $F\in Coh(Y)$, then $F(n)$ has no higher cohomology for $n$ large enough. This implies that $F(n+2)\in Mod-A$ for $n$ large enough, see Lemma \ref{lemma2}. We denote the image of $Mod-A$ in $D^b(Y)$ under the autoequivalence $\otimes\pi^*O(-n)$ by $Mod-A^n$. Then one can identify subcategories of $Coh(Y)$ with subcategories of various quiver hearts $Mod-A^n$, and further transfer the canonical orientation data for these quiver hearts in \cite{Dav} to these subcategories of $Coh(Y)$. But are these orientation data compatible? In this note, we show that the answer to this question is yes, see Theorem \ref{theorem3}. This way one obtains an orientation data for the stack of coherent sheaves by gluing the canonical orientation data on the stack of $A^n$ modules.

There is another orientation data from geometry, see section 2.2. The idea is related to Remark 5.5 in \cite{MT18}, and the general statement is explained to the author by Yukinobu Toda. One may wonder the relation between the orientation data obtained from the quiver hearts and the one from geometry, it turns out they are the same, as shown in Theorem \ref{theorem4}. 

\subsection{Outline}
In section 2, we review the definition of orientation data and some facts about the derived equivalence. In section 2.1, we review the canonical orientation data defined on the quiver heart given by \cite{Dav}. In section 2.2, we review the orientation data from geometry. 

We prove the main results in section 3. In Theorem \ref{theorem3} we show that the canonical orientation data defined in \cite{Dav} restricted to $(Mod-A^n)_c$ are compatible under the autoequivalence $\_\otimes\pi^*O(1)$. Finally we show the orientation data from the quiver heart and the one from geometry are the same in Theorem \ref{theorem4}.

\subsection{Notation}
In this note, \cblue{all shemes are assumed to be noetherian schemes over $\mathbb{C}$.} All of the $A$ modules are assumed to be right $A$ modules. For the quiver $Q$ in (\ref{equ2}), we denote the arrows between vertices by $a_i$, $b_j$, $c_k$, and the constant arrow at vertex $i$ by $e_i$. Denote $P_i$ to be the projective module of $A$ generated by $e_i$, for $i=0,1,2$. Let $S$ be an arbitrary scheme, denote the projection from $S\times Y$ to $S$ by $\pi_S$, and the projection from $S\times Y$ to $Y$ by $\pi_Y$. We denoting $R\mathcal{H}om_S(\_,\_):=R\pi_{S*}R\mathcal{H}om_{O_{S\times Y}}(\_,\_)$. Denote the projection from $Y$ to $\mathbb{P}^2$ by $\pi$, if needed we abuse notation and denote $\pi\circ \pi_Y$ also by $\pi$. Denote the inclusion from $\mathbb{P}^2$ to $Y$ by $i$.

\subsection{Acknowledgement}
This note is part of my thesis work under the direction of Sheldon Katz, I would like to thank my advisor Sheldon Katz for introducing the subject to me, and tremendous helpful discussions and encouragement. I am also very grateful to Ben Davison and Yukinobu Toda for suggesting to check the compatibility of the orientation data on quiver hearts, and explaining the material in section 2 to me. In particular, I would like to thank Ben Davison for patiently answering many questions about orientation data and motivic DT theory in general, and useful suggestions on an earlier version of the paper. I would also like to thank Davesh Maulik and Sven Meinhardt for helpful conversations. \cgreen{Finally I would like to thank the referee for carefully reading the draft and many helpful suggestions.} This work was done when I was visiting MSRI for the program Enumerative Geometry Beyond Numbers (EGN) during the spring 2018 semester, I would like to thank MSRI for the excellent working environment. 
\section{Background}

\bigskip 

Let $\mathcal{C}$ be a category, in our situation, it is either the category of right $A$ modules, or the category of coherent sheaves \ccgreen{with compact support} on $Y$. Let $St$ be the stack of objects in $\mathcal{C}$. Let $M$ be the universal object on $St$. Let $St^{(2)}$ be the stack parametrizing short exact sequences of objects in $\mathcal{C}$, and $0\rightarrow M_1\rightarrow M_2\rightarrow M_3\rightarrow 0$ be the universal short exact sequence. Then we have three maps $p_i:St^{(2)}\rightarrow St$ induced by $M_i$. \cgreen{
Note that there is an isomorphism induced by the multiplicativity of determinants on triangles and Serre duality:
\begin{equation}
\label{equKS}
\begin{split}
det(R\mathcal{H}om(M_2,M_2))\otimes det(R\mathcal{H}om(M_1,M_1))^{-1}&\otimes det(R\mathcal{H}om(M_3,M_3))^{-1}\\
&\simeq det(R\mathcal{H}om(M_1,M_3))^{\otimes 2}.
\end{split}
\end{equation}}

\cgreen{With the above information, we give the definition of orientation data:}

\begin{definition} \cite{BBBBJ15, BJM19, KS}
\label{def1}
An orientation data for $\mathcal{C}$ is a line bundle L on $St$, such that 

i. $L^2\simeq det(R\mathcal{H}om_{St}(M,M))$,

ii. $p_2^*L\simeq p_1^*L\otimes p_3^*L\otimes R\mathcal{H}om_{St^{(2)}}(p_1^*M_1, p_3^*M_3)$, \cgreen{such that the square of this isomorphism is the canonical isomorphism (\ref{equKS}).}
\end{definition}

Note that this definition is slightly different from the definition in \cite{BJM19}. In \cite{BJM19}, orientation data is defined on the reduced structure of the algebraic d-critical locus.
 
Let $T^n:=\pi^*O(n)\oplus \pi^*O(n+1)\oplus \pi^*O(n+2)$, $A^n:=End_{O_Y}(T^n)$. Note that the category of right $A^n$ modules is equivalent to $Mod-A^n$ in the introduction. As in (\ref{equ1}) we have the derived equivalence\cite{Bri05}: 
\begin{equation*}
RHom(T^n, \_): D^b(Y)\rightarrow D^b(Mod-A^n),
\end{equation*}
where the inverse map is given by $\_\otimes_{A^n}^L T^n$. 

Let $(Mod-A^n)_c$ denote the subcategory of $Mod-A^n$ consisting of modules of finite length. Since $Y$ is a noncompact CY 3-fold, we consider sheaves with compact support. Let $F$ be a coherent sheaf on $Y$ with compact support. If $RHom(T^n, F)\in Mod-A^n$ for some $n$, then $RHom(T^n, F)\in (Mod-A^n)_c$. So it is enough to consider the subcategories $(Mod-A^n)_c$. 

\bigskip

\subsection{The canonical orientation data on a quiver heart}

We first review the construction of the canonical orientation data for the stack of quiver representations. 
This section is a summary of a very small part of \cite{Dav}, and the contents in the section was explained to the author by Ben Davison during the program EGN at MSRI. The author is very grateful for his tireless explanation. The general construction can be found in \cite{Dav}, we only include here the ingredients needed for the Jacobi algebra $A^n$.    

Let $Mod_{A^n}$ be the stack of the right $A^n$ modules of finite length. Then the universal module $M(n)$ on $Mod_{A^n}$ is a $O_{Mod_{A^n}}-A^n$ bimodule, i.e. a right $O_{Mod_{A^n}}\otimes_k A^n$ module, where $O_{Mod_{A^n}}\otimes_k A^n$ is understood as a sheaf of rings on $Mod_{A^n}$. We have an isomorphism of $O_{Mod_{A^n}}$ modules (we omit $k$ in the notation). 
\begin{equation}
\label{equ4}
\begin{split}
& R\mathcal{H}om_{O_{Mod_{A^n}}\otimes A^n}(M(n),M(n))\\
&\simeq M(n)\bigotimes^L(O_{Mod_{A^n}}\otimes A^n)\bigotimes^LR\mathcal{H}om_{O_{Mod_{A^n}}\otimes A^n}(M(n), O_{Mod_{A^n}}\otimes A^n).
\end{split}
\end{equation}

\cccblue{Since $M(n)$ is flat over $Mod_{A^n}$}, from (\ref{equ4}) we know that $R\mathcal{H}om_{Mod_{A^n}\otimes A^n}(M(n), M(n))$ can be computed using a \cccblue{relative} free bimodule resolution of $O_{Mod_{A^n}}\otimes A^n$. $A^n$'s are all isomorphic to each other, we omit $n$ in the following paragraph. Since $A$ is the Jacobi algebra of a quiver algebra with potential \ccgreen{and A is Calabi-Yau \cite{Gin}}, $A$ has the following bimodule resolution \cite{Gin}:  
\begin{equation}
\label{equ5}
0\rightarrow (A\otimes_R A)^*\xrightarrow{\delta_3} A\otimes_R E^*\otimes_R A\xrightarrow{\delta_2} A\otimes_R E\otimes_R A\xrightarrow{\delta_1} A\otimes_R A\xrightarrow{\delta_0} A\rightarrow 0, 
\end{equation}
where $R\simeq e_0\mathbb{C}\oplus e_1\mathbb{C}\oplus e_2\mathbb{C}$, and $E\simeq \mathbb{C}Q_1$, where $Q_1$ is the set of arrows of $Q$. Here the $( )^*$ is the bimodule dual defined by $Hom(\_, A\otimes_R A)$. In particular, if we take the bimodule structure of $A\otimes_R A$ to be the outer bimodule structure, i.e. $a(b'\otimes b'')c=(ab')\otimes (b''c)$, then $(A\otimes_R A)^*$ is isomorphic as a bimodule to $A\otimes_R A$ with the inner bimodule structure, i.e. $a(b'\otimes b'')c=(b'c)\otimes (ab'')$, for any $b', b'', a, c\in A$.  

Each direct summand of $(A\otimes_R A)^*$ is generated by $e_i\otimes_R e_i$ as an $A\otimes_R A$ bimodule via the inner bimodule structure. 
Each direct summand of $A\otimes_R E^*\otimes_R A$ is generated by $a_i^*$, $b_j^*$, $c_k^*$ as an $A\otimes_R A$ bimodule via the inner bimodule structure. This is equivalent to generated by $e_0\otimes e_1$, $e_1\otimes e_2$, $e_2\otimes e_0$ as an $A\otimes_R A$ bimodule via the outer bimodule structure.
Each direct summand of $A\otimes_R E\otimes_R A$ is generated by $a_i$, $b_j$, $c_k$, or equivalently $e_1\otimes e_0$, $e_2\otimes e_1$, $e_0\otimes e_2$ as an $A\otimes_R A$ bimodule via the outer bimodule structure.  
Each direct summand of $A\otimes_R A$ is generated by $e_i\otimes_R e_i$ as an $A\otimes_R A$ bimodule via the outer bimodule structure. 

The map $\delta_3$ is defined by $\delta_3(e_0\otimes_R e_0)=\sum_k e_0\otimes c_k^*\otimes c_ke_2-\sum_i e_1a_i\otimes a_i^*\otimes e_0$, and similarly for $e_1\otimes_R e_1$, and $e_2\otimes_R e_2$. Or equivalently, $\delta_3(e_0\otimes_R e_0)=\sum_k c_ke_2\otimes c_k^*\otimes e_0-\sum_i e_0\otimes a_i^*\otimes e_1a_i$ via the outer bimodule structure.

$\delta_2$ is defined by $\delta_2(a_1^*)=c_3\otimes b_2\otimes e_1+e_0\otimes c_3\otimes b_2-c_2\otimes b_3\otimes e_1-e_0\otimes c_2\otimes b_3$, and similarly for $a_2, a_3, b_j, c_k$.

$\delta_1$ is defined by $\delta_1(a_i)=a_i\otimes e_0-e_1\otimes a_i$, and similarly for $b_j$, $c_k$. 

Finally $\delta_0$ is defined by multiplication on $A$. 

Note that $(O_{Mod_A}\otimes A)\otimes_{O_{Mod_A}\otimes R} (O_{Mod_A}\otimes A)$ has an $A$ bimodule structures induced by the map $A\rightarrow O_{Mod_A}\otimes A$. Apply $\_\otimes_{A^{op}\otimes_R A} (O_{Mod_A}\otimes A\otimes_{O_{Mod_A}\otimes R} O_{Mod_A}\otimes A)$ to (\ref{equ5}), we get a bimodule resolution for $O_{Mod_A}\otimes A$ (\ccgreen{to simplify notation, we abuse notation and} write $R$ for $O_{Mod_A}\otimes R$)
\begin{equation}
\label{equ40}
\begin{split}
& ((O_{Mod_A}\otimes A)\otimes_{R} (O_{Mod_A}\otimes A))^*\rightarrow (O_{Mod_A}\otimes A)\otimes_{R} E^*\otimes_{R} (O_{Mod_A}\otimes A)\rightarrow \\
& (O_{Mod_A}\otimes A)\otimes_{R} E\otimes_{R} (O_{Mod_A}\otimes A)\rightarrow (O_{Mod_A}\otimes A)\otimes_{R} (O_{Mod_A}\otimes A).
\end{split}
\end{equation}

Consider the term $R\mathcal{H}om_{O_{Mod_A}\otimes A}(M,O_{Mod_A}\otimes A)$. Recall that $Me_i$'s are vector bundles on $Mod_{A}$, denote $(Me_i)^*$'s to be the dual vector bundles. Define $M^*$ to be the left $A$ module with $(Me_0)^*\oplus (Me_1)^*\oplus (Me_2)^*$ as the underlying vector bundle, and the edges of $A$ act on it on the left via \cblue{the dual action of their action on $M$}. 

\begin{lemma}
There is an isomorphism of $A-O_{Mod_A}$ bimodules: 
\begin{equation*}
R\mathcal{H}om_{O_{Mod_{A}}\otimes A}(M, O_{Mod_{A}}\otimes A)\simeq M^*[-3].
\end{equation*} 
\end{lemma}
\begin{proof}

We show the statement by constructing a canonical isomorphism on an affine cover $\amalg S_i\rightarrow Mod_{A^n}$. Let $S=Spec(B)$ be an affine scheme. Let $M$ be a family of $A$ modules flat over $B$.

Similar to (\ref{equ40}), we have a bimodule resolution for $B\otimes A$ (\ccgreen{to simplify notation, we write $B$ for $O_S$, and $R$ for $B\otimes R$}):
\begin{equation}
\label{equ6}
\begin{split}
((B&\otimes A)\otimes_{R} (B\otimes A))^*\rightarrow (B\otimes A)\otimes_{R} E^*\otimes_{R} (B\otimes A)\rightarrow (B\otimes A)\otimes_{R} E\otimes_{R} (B\otimes A)\rightarrow \\
(B&\otimes A)\otimes_{R} (B\otimes A)\rightarrow B\otimes A. 
\end{split}
\end{equation}
Apply $M\otimes_{B\otimes A}\_$ to this resolution, we have 
\begin{equation}
\label{equ7}
(B\otimes A)\otimes_{R} M\xrightarrow{\alpha_3} (B\otimes A)\otimes_{R} E^*\otimes_{R} M\xrightarrow{\alpha_2} M\otimes_{R} E\otimes_{R} (B\otimes A)\xrightarrow{\alpha_1} M\otimes_{R} (B\otimes A)\xrightarrow{\alpha_0} M. 
\end{equation}

Since (\ref{equ6}) is an exact sequence of flat $B\otimes A$ modules, we see that (\ref{equ7}) is exact for finitely generated module $M$. \cccblue{Since $M$ is flat over $S$, (\ref{equ7}) gives a projective} resolution of $M$ as a right $B\otimes A$ module. We can compute $R\mathcal{H}om_{B\otimes A}(M,B\otimes A)$ by the resolution (\ref{equ7}).  $R\mathcal{H}om_{B\otimes A}(M,B\otimes A)$ is quasi-isomorphic to a complex of left $A$ modules:
\begin{equation}
\label{equ8}
M^*\otimes_{R} (B\otimes A)\xrightarrow{\alpha_1^\vee} M^*\otimes_{R} E^*\otimes_{R} (B\otimes A)\xrightarrow{\alpha_2^\vee} (B\otimes A)\otimes_{R} E\otimes_{R}  M^*\xrightarrow{\alpha_3^\vee} (B\otimes A)\otimes_{R} M^*.
\end{equation}

Note that if we apply $\_\otimes_{B\otimes A^{op}}M^*$ to the bimodule resolution of $B\otimes A$, we obtain a bimodule resolution for $M^*$: 
\begin{equation*}
M^*\otimes_{R} (B\otimes A)\xrightarrow{\beta_1} M^*\otimes_{R} E^*\otimes_{R} (B\otimes A)\xrightarrow{\beta_2} (B\otimes A)\otimes_{R} E\otimes_{R}  M^*\xrightarrow{\beta_3} (B\otimes A)\otimes_{R} M^*.
\end{equation*}

By a direct comparison of the maps, we have that $\alpha_1^\vee=-\beta_1$, $\alpha_2^\vee=\beta_2$ and $\alpha_3^\vee=-\beta_3$. This implies that (\ref{equ8}) is a resolution for $M^*$. Hence $R\mathcal{H}om_{B\otimes A}(M, B\otimes A)\simeq M^*[-3]$. Since the isomorphism is canonical, the statement is also true for the universal family over $O_{Mod_A}\otimes A$.
\end{proof}

By Lemma 1 and equation (\ref{equ40}), $R\mathcal{H}om_{O_{Mod_A}\otimes A}(M, M)$ is quasi-isomorphic to the complex below,
\begin{equation}
\label{equ10}
M^*\otimes_{R} M\rightarrow M^*\otimes_{R} E^*\otimes_{R} M\rightarrow M\otimes_{R} E\otimes_{R} M^*\rightarrow M\otimes_{R} M^*.
\end{equation}

Let $H$ be the bifunctor from $D^b(O_{Mod_A}\otimes A)\times D^b(O_{Mod_A}\otimes A)^{op}$ to $D^b(O_{Mod_A}\otimes k)$ defined by 
\ccgreen{
\begin{equation*}
H(N_1,N_2)=N_1\otimes^L((O_{Mod_A}\otimes A)\otimes_{R} E\otimes_{R}(O_{Mod_A}\otimes A) \rightarrow (O_{Mod_A}\otimes A)\otimes_{R} (O_{Mod_A}\otimes A))\otimes^LN_2^*.
\end{equation*} 
}
Since in the sequence (\ref{equ10}) the module in degree $i$ is dual to the module in degree $3-i$, we have

\begin{equation*}
det(R\mathcal{H}om_{O_{Mod_{A}}\otimes A}(M,M))\simeq det(H(M,M))^2.
\end{equation*}

\cgreen{
Denote $det(H(M, M))$ by $L_{A}$.
}
Since $H(M,\_)$ and $H(\_,M)$ are triangulated functors between triangulated categories, $L_A$ satisfies
\cgreen{
\begin{equation}
\label{equLA}
p_2^*L_A\simeq p_1^*L_A\otimes p_3^*L_A\otimes R\mathcal{H}om_{St^{(2)}}(p_1^*M_1, p_3^*M_3).
\end{equation}
}

\ccgreen{By definition of $H$, applying $H(M, \_)$ or $H(\_, M)$ to an exact triangle gives a triangle which is an exact sequence on each row. Since $H(M, M)$ equals to half of the complex $R\mathcal{H}om_{O_{Mod_{A}}\otimes A}(M,M)$, the square of (\ref{equLA}) satisfies the condition (ii) in definition \ref{def1}. Hence $L_A$ defines an orientation data on $Mod_{A}$.  }
\bigskip

\subsection{Orientation data on $Coh(Y)$ from geometry} 

This section was explained to the author by Yukinobu Toda during the program EGN at MSRI, the author is very grateful for his patient explanation of this unpublished work. 

To simplify notation, we will write the arguments for a single sheaf. It works exactly the same way if we replace $F$ by a family of sheaves over $Spec(B)$. Since the isomorphisms are canonical, as in lemma 1, the conclusion holds for the universal family on the stack of coherent sheaves on $Y$.  

Let $F$ be a coherent sheaf on $Y$ with compact support. We have the exact sequence 
\begin{equation}
\label{equ46}
0\rightarrow \pi^*\pi_*F(3)\xrightarrow{f_1} \pi^*\pi_*F\xrightarrow{f_2} F\rightarrow 0.
\end{equation}
Denote the action of $O_{\mathbb{P}^2}(3)$ on $\pi_*F$ by $g:O_{\mathbb{P}^2}(3)\otimes \pi_*F\rightarrow \pi_*F$. Let the map $s$ be the multiplication by the canonical section of $\pi^*O(-3)$. Then $f_1$ is explicitly given by $\pi^*g-s\otimes id$. $f_2$ is the canonical adjoint map.   
Then \cblue{we have the following exact triangle:}
\begin{equation*}
RHom(F,F)\rightarrow RHom(\pi^*\pi_*F, F)\rightarrow RHom(\pi^*\pi_*F(3),F)\rightarrow RHom(F,F)[1].
\end{equation*}

This implies that
\begin{equation*}
\begin{split}
det(RHom(F,F))
& \simeq det(RHom(\pi_*F,\pi_*F))\otimes det(RHom(\pi_*F(3), \pi_*F))^*\\
& \simeq (det(RHom(\pi_*F, \pi_*F))^2,
\end{split}
\end{equation*}
\cblue{where the second isomorphism is by Serre duality on $\mathbb{P}^2$.}
This gives a canonical square root $det(RHom(\pi_*F, \pi_*F))$ for $det(RHom(F,F))$. If we denote the stack of coherent sheaves on $Y$ with compact support by $Coh_Y$, 
the above argument gives a canonical square root of $det(R\mathcal{H}om_{Coh_Y}(E, E))$, where $E$ is the universal sheaf on $Coh_Y$. Denote this line bundle by $L_g$. 
\bigskip

\section{Orientation data on $Coh_Y$ by gluing the $L_{A^n}$'s}

\ccgreen{In this section, by a family of coherent sheaves on Y with compact support we mean that the family of supports is proper over the base.}
We first recall the following property of bounded family. To simplify notation we abuse notation and denote $\pi_Y^*(\pi^*O(i))$ by $\pi^*O(i)$, and $\pi_Y^*T^n$ by $T^n$. Let $F$ be a flat family of coherent sheaves on $Y$ with compact support over a base scheme $S$, \cblue{where $S$ is of finite type. Then for n large enough, $R^i\pi_{S*}(F\otimes \pi^*O(n))=0$ for $i>0$.} By the definition of $T^{-n}$, this implies that: 
\begin{lemma}
\label{lemma2}
\cblue{$R\mathcal{H}om_S(T^{-n-2}, F)$ is a flat family of $A^{-n-2}$ module over $S$.}
\end{lemma}
In particular we can cover the stack of coherent sheaves by the stacks of right $A^n$ modules, for $n\in\mathbb{Z}$. 

\cgreen{We first prove some technical lemmas which will be used in the main theorem.}

Let $S=Spec(B)$ be an affine scheme of finite type. Omit $n$ in the notation for simplicity, we have the following relation between the functors $R\mathcal{H}om_S(T, \_)$ and $\_\otimes^L_{B\otimes A} T$: 
\begin{lemma}
\label{lem3}
Let $M$ be a $B\otimes A$ module which is flat over $S$, then $R\mathcal{H}om_S(T, \_\otimes^L_{B\otimes A}T)\simeq id$ canonically when applied to $M$. Let $F$ be a coherent sheaf with compact support on $S\times Y$ which is flat over $S$, then $R\mathcal{H}om_S(T,\_)\otimes^L_{B\otimes A}T\simeq id$ canonically when applied to $F$.
\end{lemma}
\begin{proof}
\cccblue{Consider the first claim. Since $M$ is flat over $S$, by resolution (\ref{equ7}), $M\otimes^L_{B\otimes A}T$ is bounded below. Hence the functor $R\mathcal{H}om_S(T, \_\otimes^L_{B\otimes A}T)$ is defined. The functors $R\mathcal{H}om_S(T, \_\otimes^L_{B\otimes A}T)$ and $id$ agree canonically when applied to free $B\otimes A$ modules and modules of the form $P\otimes A$, where $P$ is a direct summand of free copies of $B$. As a result, these two functors agree for any finitely generated $B\otimes A$ module $M$ which is flat over $S$. }   

\cred{On the other hand, let $F$ be a family of coherent sheaves on $Y$ with compact support which is flat over $S$. Then $\pi_*F$ is a flat family of coherent sheaves on $\mathbb{P}^2$. The argument then follows from the well known result that $D^b(\mathbb{P}^2)$ is generated by $O$, $O(1)$ and $O(2)$. Denote $\pi_1: S\times \mathbb{P}^2\times \mathbb{P}^2\rightarrow S\times \mathbb{P}^2$, and $\pi_2:S\times \mathbb{P}^2\times \mathbb{P}^2\rightarrow S\times \mathbb{P}^2$ to be the projection to the first and second factor of $\mathbb{P}^2\times \mathbb{P}^2$. Denote $\pi_{\mathbb{P}^2}^i:S\times\mathbb{P}^2\rightarrow \mathbb{P}^2$ to be the corresponding projections to $\mathbb{P}^2$, and $\pi_S^i:S\times\mathbb{P}^2\rightarrow S$ be the corresponding projection to $S$. Recall Beilinson's resolution of the diagonal:
\cgreen{
\begin{equation*}
0\to... \to\pi_1^*(\pi_{\mathbb{P}^2}^{1*}\Omega^i(i))\otimes\pi_2^*(\pi_{\mathbb{P}^2}^{2*} O(-i))\to...\to O_{S\times\mathbb{P}^2\times\mathbb{P}^2}\to O_{S\times\Delta}.
\end{equation*}
}
We see that $\pi_*F$ is generated by 
\begin{equation*}
\Phi^i:=R\pi_{1*}(\pi_2^*(\pi_*F)\otimes^L(\pi_1^*(\pi_{\mathbb{P}^2}^{1*}\Omega^i(i))\otimes_{O_{S\times\mathbb{P}^2\times \mathbb{P}^2}} \pi_2^*(\pi_{\mathbb{P}^2}^{2*} O(-i)))
\end{equation*} 
for $i=0, 1, 2$. By the projection formula, $\Phi^i\simeq (\pi_{\mathbb{P}^2}^{1*}\Omega^i(i))\otimes_{O_{S\times \mathbb{P}^2}} R\pi_{1*}(\pi_2^*(\pi_*F(-i)))$. Since $\pi_*F(-i)|_s$ has no higher cohomology for any point $s\in S$, we have 
\begin{equation*}
\begin{split}
\Phi^i&\simeq (\pi_{\mathbb{P}^2}^{1*}\Omega^i(i))\otimes_{O_{S\times\mathbb{P}^2}}\pi_{1*}(\pi_2^*\pi_*F(-i))\\
&\simeq (\pi_{\mathbb{P}^2}^{1*}\Omega^i(i))\otimes_{O_{S\times\mathbb{P}^2}}\pi_S^{1*} \pi^2_{S*}(\pi_*F(-i)).
\end{split}
\end{equation*}
Since $\pi_*F$ is flat over $S$, $\Gamma(\pi^2_{S*}(\pi_*F(-i))$ is a projective $B$ module. Since $\Omega^i(i)$ is quasi-isomorphic to a finite complex of $O, O(1), O(2)$, we have $\Phi^i$, and hence $\pi_*F$ is generated by objects $\pi_{\mathbb{P}^2}^{1*}O(i)\otimes_{O_{S\times\mathbb{P}^2}} \pi_S^{1*}\widetilde{Q}$, for $Q$ a projective $B$ module \cgreen{and $\widetilde{Q}$ the sheaf associated to the module $Q$}. By the exact sequence (\ref{equ46}), we have $F$ is generated by objects $\pi_Y^{1*}(\pi^*O(i))\otimes_{O_{S\times Y}} \pi_S^{1*}\widetilde{Q}$. Since $Q$ is a direct summand of free copies of $B$, $R\mathcal{H}om_S(T,\_)\otimes^L_{B\otimes A}T\simeq id$ canonically when applied to $\pi_Y^{1*}(\pi^*O(i))\otimes_{O_{S\times Y}} \pi_S^{1*}\widetilde{Q}$, hence the conclusion follows. 
}

\end{proof}

\ccgreen{
Let $Coh^n:=Coh_Y\cap Mod_{A^{-n}}$, $Coh_Y=\cup Coh^n$.} \cgreen{Recall that $E$ is the universal sheaf on $Coh_Y$ defined at the end of section 2.2}. On $Coh^n$, we have 
\begin{lemma}
\label{lem4}
\begin{equation}
\label{equlem4}
det(R\mathcal{H}om_{Mod_{A^{-n}}}(M(-n), M(-n)))|_{Coh^n}\simeq det(R\mathcal{H}om_{Coh_Y}(E,E))|_{Coh^n}.
\end{equation}
\end{lemma}
\begin{proof}
It is enough to show there is such a canonical isomorphism over an affine scheme $S=Spec(B)$. Let $F$ be a coherent sheaf on $S\times Y$ which is flat over $S$, such that $R\mathcal{H}om_S(T^{-n},F)$ is a flat family of $A^{-n}$ modules over $S$. 

By lemma \ref{lem3} we have
\begin{equation*}
F\simeq R\mathcal{H}om_S(T^{-n}, F)\otimes^L_{O_S\otimes A^{-n}}T^{-n}.
\end{equation*} 
Since $R\mathcal{H}om_S(T^{-n}, F)$ is a finitely generated $B\otimes A^{-n}$ module flat over $B$, it has a projective resolution given by (\ref{equ7}), denote it by $N^\bullet$. This implies that $F$ has a resolution consisting of copies of $V_i\otimes_{R}T^{-n}$, $V_i\otimes_{R}E\otimes_{R}T^{-n}$ or their duals \ccgreen{for $V_i$ some $R$-module}, denote it by $L^\bullet$. 
Then we have the following isomorphisms:
\begin{equation}
\label{equF}
R\mathcal{H}om_S(F, F)\simeq R\mathcal{H}om_S(L^\bullet, F),
\end{equation} 
\begin{equation}
\label{equT}
RHom_{O_S\otimes A^{-n}}(R\mathcal{H}om_S(T^{-n}, F), R\mathcal{H}om_S(T^{-n}, F))\simeq RHom_{O_S\otimes A^{-n}}(N^\bullet, R\mathcal{H}om_S(T^{-n}, F)).
\end{equation}
Since $R\mathcal{H}om_S(L^i, F)$ has no higher cohomology, we see the complexes of the RHS of (\ref{equF}) and (\ref{equT}) are canonically isomorphic. 
Hence we have a canonical isomorphism: 
\begin{equation*}
R\mathcal{H}om_S(F, F)\simeq RHom_{O_S\otimes A^{-n}}(R\mathcal{H}om_S(T^{-n}, F), R\mathcal{H}om_S(T^{-n}, F)).
\end{equation*}

Hence the statement of the lemma follows. 

\end{proof}

\begin{theorem}
\label{theorem3}
$L_{A^n}$'s are compatible and glue to be a line bundle on $\cup Mod_{A^n}$. In particular, this gives an orientation data on $Coh_Y$.
\end{theorem}

\begin{proof}
Consider the hearts $Mod-A^0$ and $Mod-A^1$. 

Let $S=Spec(B)$ be an affine scheme of finite type. 
Let $M$ be a flat family of right $A^0$ modules over $S$ of finite length. Then $F:=M\otimes_{B\otimes A^0}\pi^*T^0$ is a family of objects in $D^b(Y)$ over $S$. $M$ can also be viewed as a family of objects in $D^b(Mod-A^1)$,  assuming $M$ is also a family of $A^1$ modules. Note that $A^0$ and $A^1$ has $C=End(\pi^*(O(1)\oplus O(2)))$ as a subalgebra, and the restrictions of the two module structures of $M$ to $C$ agree. 

We denote $M$ by $M_0$ as a $B\otimes A^0$ module, and $V_0\oplus V_1\oplus V_2$ as a $B\otimes R^0$ module. Similarly we denote $M$ by $M_1$ as a $B\otimes A^1$ module, and $W_0\oplus W_1\oplus W_2$ as a $B\otimes R^1$ module. 
We have that
\begin{equation*}
V_0\simeq R\mathcal{H}om_S(O, F),
\end{equation*}
\begin{equation}
\label{equV1}
V_1\simeq R\mathcal{H}om_S(\pi^*O(1), F)\simeq W_0,
\end{equation}
\begin{equation}
\label{equV2}
V_2\simeq R\mathcal{H}om_S(\pi^*O(2), F)\simeq W_1,
\end{equation}
\begin{equation*}
W_2\simeq R\mathcal{H}om_S(\pi^*O(3), F).
\end{equation*}

We have the exact sequence 
\begin{equation}
\label{equ20}
0\rightarrow \pi^*O(-3)\rightarrow \pi^*O(-2)^3\rightarrow \pi^*O(-1)^3\rightarrow O\rightarrow 0,
\end{equation} 
which gives two exact triangles 
\begin{equation}
\label{equ21}
\pi^*O(-3)\rightarrow \pi^*O(-2)^3\rightarrow Q\rightarrow\pi^*O(-3)[1]
\end{equation}
and 
\begin{equation}
\label{equ22}
Q\rightarrow \pi^*O(-1)^3\rightarrow O\rightarrow Q[1].
\end{equation} 

\ccgreen{Note that $Q\simeq \pi^*\Omega_{\mathbb{P}^2}$, hence is locally free.} Apply $\_\otimes^LF$ to (\ref{equ21}) and (\ref{equ22}), since every term in the above triangles are locally free, we get exact triangles
\begin{equation*}
F(-3)\rightarrow F(-2)^3\rightarrow Q\otimes F\rightarrow F(-3)[1],
\end{equation*} 

\begin{equation*}
Q\otimes F\rightarrow F(-1)^3\rightarrow F\rightarrow Q\otimes F[1].
\end{equation*}
Apply $R\mathcal{H}om_S(O,\_)$ to the above two exact triangles. Since $R\mathcal{H}om_S(\pi^*T^0,F)$ is a family of $A^0$ modules, and $R\mathcal{H}om_S(\pi^*T^1,F)$ is a family of $A^1$ modules, we have $R^i\mathcal{H}om_S(O, F(-k))=0$ for all $i\neq 0$, $k=0,1,2,3$. By the long exact sequences, the only non zero cohomology of $Q\otimes F$ is $R^0\mathcal{H}om_S(O, Q\otimes F)$. Thus we get an exact sequence of vector bundles over $S$
\begin{equation*}
0\rightarrow R\mathcal{H}om_S(O, F(-3))\rightarrow R\mathcal{H}om_S(O, F(-2))^3\rightarrow R\mathcal{H}om_S(O, F(-1))^3\rightarrow R\mathcal{H}om_S(O, F)\rightarrow 0.
\end{equation*}
This is equivalent to 
\begin{equation}
\label{equW}
0\rightarrow W_2\rightarrow W_1^3\rightarrow W_0^3\rightarrow V_0\rightarrow 0.
\end{equation}

Now we compute
\begin{equation*}
det(H(M,M))\simeq det(M\otimes_{R_B}E\otimes_{R_B}M^*)^*\otimes det(M\otimes_{R_B} M^*)
\end{equation*}
both as $A^0$ module, and as $A^1$ module. We have 
\begin{equation}
\label{equM1}
\begin{split}
det(H(M_1,M_1))\simeq 
& (det(W_1)^{d_0}\otimes det(W_0^*)^{d_1})^3 \\
& \otimes (det(W_2)^{d_1}\otimes det(W_1^*)^{d_2})^3 \\
& \otimes (det(W_0)^{d_2}\otimes det(W_2^*)^{d_0})^3.
\end{split}
\end{equation}

While as $A^0$ module, we have the isomorphism \cgreen{induced by (\ref{equV1}), (\ref{equV2}) and (\ref{equW})}:
\begin{equation}
\label{equM0}
\begin{split}
det(H(M_0,M_0))\simeq 
& (det(W_0)^{3d_0-3d_1+d_2}\otimes ((det(W_0)^3\otimes det(W_1^*)^3\otimes det(W_2))^*)^{d_0})^3\\
& \otimes (det(W_1)^{d_0}\otimes det(W_0^*)^{d_1})^3\\
& \otimes ((det(W_0)^3\otimes det(W_1^*)^3\otimes det(W_2))^{d_1}\otimes (det(W_1^*)^{3d_0-3d_1+d_2})^3\\
\simeq & (det(W_0)^{-3d_1+d_2}\otimes det(W_1)^{3d_0}\otimes det(W_2^*)^{d_0})^3\\
& \otimes (det(W_1)^{d_0}\otimes det(W_0^*)^{d_1})^3\\
& \otimes (det(W_0)^{3d_1}\otimes det(W_2)^{d_1}\otimes (det(W_1^*)^{3d_0+d_2})^3\\
\simeq & (det(W_0)^{d_2}\otimes det(W_2^*)^{d_0})^3 \otimes (det(W_0)^{-3d_1}\otimes det(W_1)^{3d_0})^3\\
& \otimes (det(W_1)^{d_0}\otimes det(W_0^*)^{d_1})^3\\
& \otimes (det(W_2)^{d_1}\otimes det(W_1^*)^{d_2})^3 \otimes(det(W_0)^{3d_1}\otimes det(W_1)^{-3d_0})^3.
\end{split}
\end{equation}

Hence we have a canonical isomorphism \cgreen{induced by (\ref{equM1}) and (\ref{equM0})}:
\begin{equation*}
det(H(M_1,M_1))\simeq det(H(M_0,M_0)).
\end{equation*}

In general, assume that $M$ is a family of right $A^0$ module, which is also a family of right $A^m$ module. Since $M$ is a family of right $A^m$ module, let $F:=M\otimes_{B\otimes A^m}\pi^*T^m$. As in (\ref{equ21}) and (\ref{equ22}), we have exact triangles
\begin{equation}
\label{equ31}
F(-m-2)\rightarrow F(-m-1)^3\rightarrow Q\otimes F\rightarrow F(-m-2)[1]
\end{equation}
and 
\begin{equation}
\label{equ32}
Q\otimes F\rightarrow F(-m)^3\rightarrow F(-m+1)\rightarrow Q\otimes F[1].
\end{equation}
By (\ref{equ31}), we know that $R^iHom(O, Q\otimes F)=0$ for $i>0$, so by (\ref{equ32}) $R^iHom(O, F(-m+1))=0$ for $i>0$. This implies that $R^iHom(T^k, F)=0$ for all $k<m$, $i>0$. 
Similarly since $M$ is a family of right $A^0$ modules, by (\ref{equ22}) we know that $R^iHom(O, Q\otimes F)=0$ for $i<0$, so by (\ref{equ21}) we know that $R^iHom(O, F(-3))=0$ for $i<0$. This implies that $R^iHom(T^k, F)=0$ for all $k>0$, $i<0$. Combining the above observation, $M$ is also a family of right $A^k$ module, for $0<k<m$. We have canonical isomorphisms
\begin{equation*}
det(H(M_m,M_m))\simeq det(H(M_{m-1},M_{m-1}))\simeq ...\simeq det(H(M_1,M_1))\simeq det(H(M_0,M_0)).
\end{equation*}

The above isomorphisms are \cgreen{induced by repeatedly applying (\ref{equ20})}, hence the $L_{A^n}$'s glue to be a line bundle $L_a$ on $\cup Mod_{A^n}$. 

By repeatedly applying relation (\ref{equ20}) and the similar argument for gluing $L_{A^n}$'s, we have $det(R\mathcal{H}om_{Mod_{A^{-n}}}(M(-n), M(-n)))$'s glue to be a line bundle $L_M$. By (\ref{equlem4}), we have $L_M|_{Coh_Y}\simeq det(R\mathcal{H}om_{Coh_Y}(E,E))$. Hence we have $(L_a|_{Coh_Y})^2\simeq L_M|_{Coh_Y}\simeq det(R\mathcal{H}om_{Coh_Y}(E, E))$.

Leting $p_{13}:M^{(2)}\rightarrow Coh_Y\times Coh_Y$, then \ccgreen{$M^{(2)}=\cup(p_{13}^{-1}(Coh^n\times Coh^n))$}. \cgreen{As discussed at the end of Section 2.1,} $L_a|_{Coh^n}$ satisfies the compatibility condition (ii) in Definition \ref{def1} on $p_{13}^{-1}(Coh^n\times Coh^n)$. Furthermore, the isomorphism is canonical for all $n$, so $L_a$ defines an orientation data on $Coh_Y$. 

\end{proof}

\begin{theorem}
\label{theorem4}
$L_g\simeq L_a|_{Coh_Y}$
\end{theorem}
\begin{proof}
Since $Coh_Y= \cup Coh^n$, it is enough to show there is a canonical isomorphism $L_g|_{Coh^n}\simeq L_a|_{Coh^n}$. WLOG assume that $n=0$. Let $T^0_{\mathbb{P}^2}:=O_{\mathbb{P}^2}\oplus O_{\mathbb{P}^2}(1)\oplus O_{\mathbb{P}^2}(2)$, and $A^0_{\mathbb{P}^2}$ be $End_{\mathbb{P}^2}(T_{\mathbb{P}^2}^0)$. 
We write (\ref{equ5}) \cite{Gin} in the following explicit form:
\begin{equation*}
0\rightarrow (A\otimes_R A)^*\xrightarrow{\delta_3} 
\begin{array}{c} \oplus_i A\otimes a_i^* \otimes A\\ 
\oplus_j A\otimes b_j^* \otimes A \\ 
\oplus_k A\otimes c_k^* \otimes A 
\end{array}
\xrightarrow{\delta_2} 
\begin{array}{c} \oplus_i A\otimes a_i \otimes A\\ 
\oplus_j A\otimes b_j \otimes A \\ 
\oplus_k A\otimes c_k \otimes A 
\end{array}
\xrightarrow{\delta_1} A\otimes_R A\xrightarrow{\delta_0} A\rightarrow 0. 
\end{equation*}

The resolution for $A^0_{\mathbb{P}^2}$ can be identified with part of the resolution for $A$, we omit $0$ in the notation: 
\begin{equation*}
\cblue{
0\xrightarrow{\eta_3} \oplus_k A_{\mathbb{P}^2}\otimes_R c_k^*\otimes_R A_{\mathbb{P}^2}\xrightarrow{\eta_2} 
\begin{array}{c} \oplus_i A_{\mathbb{P}^2}\otimes a_i \otimes A_{\mathbb{P}^2}\\ 
\oplus_j A_{\mathbb{P}^2}\otimes b_j \otimes A_{\mathbb{P}^2} \\ 
\end{array}
\xrightarrow{\eta_1} A_{\mathbb{P}^2}\otimes_R A_{\mathbb{P}^2}\xrightarrow{\eta_0} A_{\mathbb{P}^2}\rightarrow 0}.
\end{equation*}
Note that here $c_k^*$ is just a notation for the generator $e_2\otimes e_0$ as an $A\otimes A$ bimodule with the outer bimodule structure. Let $F$ be a flat family of coherent sheaves on $Y$ with compact support over an affine scheme $S=Spec(B)$, assume that $R\mathcal{H}om_S(\pi^*T^0, F)$ is a $B\otimes A^0$ module, denote it by $M$. As a $B\otimes R$ module, denote $M$ by $V_0\oplus V_1\oplus V_2$. Then 
\begin{equation*}
\begin{split}
R\mathcal{H}om_S(F, F)
& \simeq R\mathcal{H}om_{B\otimes A^0}(M, M) \\
&\simeq \oplus_i V_i^*\otimes V_i \rightarrow 
\begin{array}{c} \oplus_i V_1^*\otimes a_i^* \otimes V_0\\ 
\oplus_j V_2^*\otimes b_j^* \otimes V_1 \\ 
\oplus_k V_0^*\otimes c_k^* \otimes V_2 
\end{array} \rightarrow 
\begin{array}{c} \oplus_i V_1\otimes a_i \otimes V_0^*\\ 
\oplus_j V_2\otimes b_j \otimes V_1^* \\ 
\oplus_k V_0\otimes c_k \otimes V_2^* \\
\end{array}\rightarrow 
\oplus_i V_i\otimes V_i^*.
\end{split}
\end{equation*}
Then $\pi_*M:= R\mathcal{H}om_S(\pi_{\mathbb{P}^2}^*T_{\mathbb{P}^2}^0, (id_S\times\pi)_*F)$ is a $B\otimes A^0_{\mathbb{P}^2}$ module, via the inclusion of $A^0_{\mathbb{P}^2}$ in $A^0$. As a $B\otimes R$ module, $\pi_*M\simeq V_0\oplus V_1\oplus V_2$.  
\begin{equation*}
\begin{split}
R\mathcal{H}om_S(\pi_*F, \pi_*F)
&\simeq R\mathcal{H}om_{B\otimes A^0_{\mathbb{P}^2}}(\pi_*M, \pi_*M)\\
&\simeq 0\rightarrow 
\begin{array}{c} 
\oplus_k V_0^*\otimes c_k^* \otimes V_2 
\end{array} \rightarrow 
\begin{array}{c} \oplus_i V_1\otimes a_i \otimes V_0^*\\ 
\oplus_j V_2\otimes b_j \otimes V_1^* \\ 
\end{array}\rightarrow 
\oplus_i V_i\otimes V_i^*.
\end{split}
\end{equation*}

Since $det(\oplus_k V_0^*\otimes c_k^* \otimes V_2)\simeq det(\oplus_k V_0\otimes c_k \otimes V_2^*)^*$, 
$det(R\mathcal{H}om_S(\pi_*F, \pi_*F))\simeq L_a|_S$. Since the isomorphism is canonical, we have $det(R\mathcal{H}om_{Coh^0}(\pi_*F, \pi_*F))\simeq L_a|_{Coh^0}$

This isomorphism is canonical for all n, so we actually have $L_g\simeq L_a|_{Coh_Y}$. 

\end{proof}


\bigskip
\bigskip

\begin{thebibliography}{99}

\bibitem[BBBBJ15]{BBBBJ15}
Oren Ben-Bassat, Christopher Brav, Vittoria Bussi, Dominic Joyce, \emph{A 'Darboux theorem' for shifted symplectic structures on derived Artin stacks, with applications,} Geom. Topol. 19 (2015), no. 3, 1287-1359. 

\bibitem[Bri05]{Bri05}
Tom Bridgeland, \emph{t-structures on some local Calabi-Yau varieties,} Journal of Algebra, 289-2, 2005, 453-483

\bibitem[Bur04]{Bur04} 
Michel Van den Bergh, \emph{Three-dimensional flops and non-commutative rings,} Duke Mathematical Journal 
122-3 (2004), 423-455

\bibitem[BJM19]{BJM19} 
Vittoria Bussi, Dominic Joyce, Sven Meinhardt, \emph{On motivic vanishing cycles of critical loci,} Journal of Algebraic Geometry, 28 (2019), 405-438 

\bibitem[Dav]{Dav} 
Ben Davison, \emph{Invariance of orientation data for ind-constructible Calabi-Yau $A_\infty$ categories under derived equivalence,} preprint, arXiv:1006.5475   

\bibitem[Gin]{Gin} 
Victor Ginzburg, \emph{Calabi-Yau Algebras,} preprint, arXiv:math/0612139   

\bibitem[KS]{KS} 
Maxim Kontsevich, Yan Soibelman, \emph{Stability structures, motivic Donaldson-Thomas invariants and cluster transformations,} preprint,	arXiv:0811.2435   

\bibitem[Joy15]{Joy15} 
D. Joyce, \emph{A classical model for derived critical loci,} J. Diff. Geom. 101(2015), 289-367.	
	
\bibitem[MT18]{MT18} 
Davesh Maulik, Yukinobu Toda, \emph{Gopakumar-Vafa invariants via vanishing cycles,} Inventiones mathematicae, 1-81  


\bibitem[Tod]{Tod}
Yukinobu Toda, \emph{Gopakumar-Vafa invariants and wall-crossing,} preprint, arXiv:1710.01843


\end{thebibliography}
\end{document}